\documentclass[a4paper,11pt]{amsart}

\usepackage[utf8x]{inputenc}
\usepackage[english]{babel}
\usepackage{geometry}
\usepackage{fullpage}
\usepackage{hyperref}
\usepackage{amsmath}
\usepackage{amsthm}
\usepackage{amsfonts}
\usepackage{amssymb}
\usepackage{amsxtra}
\usepackage{graphicx}
\usepackage{epsfig}
\usepackage{enumerate}
\usepackage{mathrsfs}
\usepackage{xfrac}
\usepackage{tikz}
\usetikzlibrary{arrows}
\usepackage{latexsym}
\usepackage{titletoc}
\usepackage{indentfirst}
\usepackage{stmaryrd}

\swapnumbers
\theoremstyle{plain}
\newtheorem{prop}{Proposition}[section]
\newtheorem{theorem}[prop]{Theorem}

\newtheorem{lemma}[prop]{Lemma}
\newtheorem{crit}[prop]{Useful Criterion}

\theoremstyle{definition}
\newtheorem{definit}[prop]{Definition}
\newtheorem{parag}[prop]{}
\theoremstyle{remark}
\newtheorem{remark}[prop]{\textsc{Remark}}
\newtheorem{notat}[prop]{\textsc{Notation}}

\newcommand{\Art}{\text{Art}}

\newcommand{\qcoh}{\text{qcoh}}

\newcommand{\Hom}{\text{Hom}}
\newcommand{\id}{\text{id}}
\newcommand{\ob}{\text{ob}}

\newcommand{\bbA}{\mathbb{A}}

\newcommand{\calX}{\mathcal{X}}

\newcommand{\calY}{\mathcal{Y}}
\newcommand{\frakC}{\mathfrak{C}}

\newcommand{\frakE}{\mathfrak{E}}
\newcommand{\frakN}{\mathfrak{N}}
\newcommand{\Ho}{\text{Ho}}
\newcommand{\opp}{\text{opp}}

\newcommand{\lisse}{\text{lis-\'et}}
\newcommand{\barx}{\overline{x}}
\newcommand{\bark}{\overline{k}}

\newcommand{\scrO}{\mathscr{O}}

\newcommand{\virt}{\text{virt}}

\newcommand{\AbCones}{(\sfrac{\text{AbCones}}{\mathcal{X}})}
\newcommand{\CCoh}{\text{C}^{[-1,0]}({\text{Coh}}(\mathcal{X}_{\text{lis-\'et}}))}
\newcommand{\Ccoh}{\text{C}_{\text{coh}}^{[-1,0]}(\mathcal{X}_{\text{lis-\'et}})}
\newcommand{\DCoh}{\text{D}^{[-1,0]}({\text{Coh}}(\mathcal{X}_{\text{lis-\'et}}))}
\newcommand{\Dcoh}{\text{D}_{\text{coh}}^{[-1,0]}(\mathcal{X}_{\text{lis-\'et}})}
\newcommand{\Dcohprime}{\text{D}_{\text{coh}}^{[-1,0]}(\mathcal{X}_{\text{lis-\'et}}')}
\newcommand{\Onecoh}{\text{D}_{\text{coh}}^{[-1,1]}(\mathcal{X}_{\text{lis-\'et}})}
\renewcommand{\phi}{\varphi}

\DeclareMathOperator{\aut}{Aut}

\DeclareMathOperator{\im}{im}

\DeclareMathOperator{\rk}{rk}
\DeclareMathOperator{\spec}{Spec}
\DeclareMathOperator{\sez}{\Gamma}

\author{Flavia Poma}
\title{Virtual classes of Artin stacks}
\date{\today\\
MSC classes: 14C15, 14N35}

\begin{document}
\begin{abstract}
  We construct virtual fundamental classes of Artin stacks over a Dedekind domain endowed
  with a perfect obstruction theory.
\end{abstract}
\maketitle
\setcounter{tocdepth}{1}
\tableofcontents
\section{Introduction}
Virtual classes of moduli stacks play a central role in enumerative
geometry as they represent a major ingredient in the
construction of deformation invariants (e.g. Gromov-Witten,
Donaldson-Thomas). They were introduced in \cite{BF} for
Deligne-Mumford stacks; the construction has been applied to develop
Gromov-Witten theory of algebraic varieties and Deligne-Mumford
stacks, as well as Donaldson-Thomas theory (see \cite{B}, \cite{BM},
\cite{AGV02}, \cite{AGV08}, \cite{Tho}).

\subsection{The existing construction for DM stacks} A
  virtual class of an algebraic stack $\calX$ is a functorial cycle
  class ${[\calX]}^{\virt}$ in the Chow group $A_*(\calX)$ with
  rational coefficients. The idea of the construction comes from
  observing the following two facts of intersection theory of schemes
  and relies on the correspondence between abelian cone stacks and
  2-term complexes of coherent sheaves
  (Theorem~\ref{theorem_deligne}).
\begin{enumerate}
\item Whenever a scheme $X$ is pure dimensional of pure dimension $N$,
  we can define a \emph{fundamental class} in $A_N(X)$ as $[X]=\sum_i
  m_i[X_i]$, where the sum is over the irriducible components of $X$
  and $m_i$ is the multiplicity of $X_i$ in $X$. If $i\colon X
  \rightarrow M$ is a closed immersion and there exists a closed
  immersion $j\colon C_i\rightarrow E$ of the normal cone $C_i$ in a
  vector bundle $\pi\colon E\rightarrow X$ (e.g., if $i$ is regular
  then the normal sheaf $N_i$ is a vector bundle) then $[X]=0^![C_i]$,
  where $0^!$ is the inverse of $\pi^*$ (\cite{fulton} 6.2.1).
\item Let $f\colon X \rightarrow Y$ be a morphism of schemes which
  factors as $p\circ i$ with $i$ closed immersion and $p$ smooth. If
  $f$ is a closed immersion then $N_f=\sfrac{N_i}{i^*T_p}$ and
  $C_f=\sfrac{C_i}{i^*T_p}$, where $T_p$ is the tangent bundle. In
  general, the morphism $i^*T_p\rightarrow N_i$ is not injective
  (\cite{fulton} B.7.5).
\end{enumerate}
The idea is to generalize (1) to morphisms of stacks. For this we need
to construct an \emph{intrinsic normal cone} $\frakC_f$ associated to
a morphism $f\colon \calX\rightarrow \calY$ of algebraic stacks and a
closed immersion $\frakC_f\rightarrow \frakE_f$ in a vector bundle
stack $\frakE_f$ over $\calX$. Locally $f$ factors as $p\circ i$, then we can define $\frakC_f$
locally as the quotient stack $[\sfrac{C_i}{i^*T_p}]$ and then glue
these together. To make the gluing process feasible, we construct an
\emph{intrinsic normal sheaf} $\frakN_f$ such that $\frakC_f$ is a
subcone of $\frakN_f$. The stack $\frakN_f$ is obtained globally via the
correspondence mentioned above as the stack associated to the
cotangent complex $L_f^\bullet$ of $f$ (\cite{laumon}, \cite{olsson},
\cite{LO}). Finally, the existence of a closed immersion
$\frakC_f\rightarrow \frakE_f$ is reduced, via the correspondence
above, to the existence of a \emph{perfect obstruction theory}.
\subsection{Obstructions for Artin stacks}
  The cotangent complex of an Artin stack has three terms, so that one
  cannot exploit directly the above correspondence to get the
  intrinsic normal sheaf. A first step in this direction was done by
  Francesco Noseda in his PhD thesis (\cite{noseda}), even though his
  construction was not completely proven to be \emph{intrinsic} and
  therefore may depend on the chosen resolution of the perfect
  obstruction theory (this point is crucial to prove the functoriality
  of the virtual fundamental class). To our knowledge, no construction
  of the virtual fundamental class of an Artin stack has been done so
  far.
\subsection{Virtual classes of Artin stacks}
  Let $f\colon \calX \rightarrow \calY$ be a morphism of algebraic
  stacks. We define the intrinsic normal sheaf $\frakN_f$ as the cone
  stack associated to the truncation $\tau_{[-1,0]}L_f^\bullet$ of the
  cotangent complex of $f$ (\ref{def_normal_sheaf}). We notice that
  there exists a smooth atlas $u\colon U \rightarrow \calX$ such that
  $f\circ u=p\circ i$ (\ref{remark_factorization}), then
  \[u^*\tau_{[-1,0]}L_f^\bullet=[\sfrac{\mathscr{I}}{\mathscr{I}^2}\rightarrow
  \ker(i^*\Omega_p\rightarrow \Omega_u)]\text{,}\] and hence
  $u^*\frakN_f=[\sfrac{N_i}{(\sfrac{i^*T_p}{T_u})}]$. As a consequence
  we can define the intrinsic normal cone $\frakC_f$, locally in the lisse-\'etale topology, as
  $[\sfrac{C_i}{(\sfrac{i^*T_p}{T_u})}]$
  (Proposition~\ref{prop_normal_cone}).  If $f$ is of Deligne-Mumford
  type then we can assume $u$ to be \'etale, therefore $\Omega_u=0$
  and our construction of the intrinsic normal cone above coincides
  with the one in \cite{BF}. Regarding obstruction theories, we prove
  an infinitesimal criterion based on the deformation theory of
  morphisms of stacks (Theorem~\ref{theorem_critot}). It should be
  mentioned that another crucial tool for the construction is
  intersection theory for Artin stacks: Chow groups for Artin stacks
  over a field are defined in \cite{kresch} and we verified that
  Kresch's theory naturally extends to stacks over a Dedekind
  domain. As a consequence we get that Manolache's construction of the
  virtual pullback in~\cite{manolache} applies to morphisms of Artin
  stacks over a Dedekind domain. In \cite{manolache} Manolache uses
  the virtual pullback to give a short proof of Costello's pushforward
  formula (\cite{costello} 5.0.1). Here we apply Manolache's
  construction to prove the pushforward formula in a more general
  setting (Proposition~\ref{prop_costellopf}).
\subsection{Possible applications}
The main motivation for this work comes from the following
applications: Gromow-Witten invariants of tame Artin stacks and
Donaldson-Thomas invariants in presence of semi-stable sheaves. In
both cases the relevant moduli stack is algebraic but not
Deligne-Mumford, therefore the construction presented here provides an
associated virtual fundamental class. We hope to return on these
points in a future paper.
\subsection*{Acknowledgements}
I would like to thank Prof. Burt Totaro for pointing out a mistake in
Remark~\ref{remark_resolution}: his valuable comments helped me
removing from section~\ref{section_ins} the assumption of $\calX$
admitting a stratification by global quotients.
\subsection*{Notations}
All stacks are Artin stacks in the sense of \cite{artin},
\cite{laumon} and are of finite type over a Dedekind domain. If
$\calX$ is an algebraic stack, we denote by $\calX_{\lisse}$ the
lisse-\'etale topos of $\calX$. If $F$ is a sheaf on $\calX_{\lisse}$,
we denote by $C(F)$ the associated abelian cone over $\calX$.
\section{The intrinsic normal cone}\label{section_ins}
\begin{parag}{\textbf{Abelian cone stacks.}}
  Let $\calX$ be an algebraic stack. A \emph{cone stack} over $\calX$
  is an algebraic $\calX$-stack $\frakC$ (together with a section and
  an $\bbA^1$-action) such that, lisse-\'etale locally on $\calX$, $\frakC$
  is the quotient stack $[\sfrac{C}{E}]$ of a cone $C$ by a vector
  bundle $E$ on $\calX$. A morphism of cone stacks is an
  $\bbA^1$-equivariant morphism of $\calX$-stacks. A $2$-isomorphism
  of cone stacks is an $\bbA^1$-equivariant $2$-isomorphism.  A cone
  stack $\frakC$ is \emph{abelian} if, lisse-\'etale locally on $\calX$, $\frakC \cong
  [\sfrac{C}{E}]$, where $C$ is an abelian cone. An abelian cone stack
  $\frakC$ is a \emph{vector bundle stack} if, lisse-\'etale locally on $\calX$, $\frakC
  \cong [\sfrac{C}{E}]$, where $C$ is a vector bundle (for further
  details see \cite{BF} 1.8--1.9). We denote by $\AbCones$ the
  $2$-category of abelian cone stacks over $\calX$.
\end{parag}
\begin{notat}
  Let $\CCoh$ be the category of complexes of coherent sheaves in the
  topos $\calX_{\lisse}$ with cohomology sheaves concentrated in
  degree $-1$ and $0$. Let $\Ccoh$ be the category of complexes of
  sheaves in the topos $\calX_{\lisse}$ with coherent cohomology
  concentrated in degrees $-1$, $0$. We can view $\CCoh$ and $\Ccoh$
  as $2$-categories, where the $2$-morphisms are homotopies (\cite{BF}
  2); we denote by $\DCoh$ and $\Dcoh$ the associated derived categories.
\end{notat}
\begin{remark}\label{remark_resolution}
  Let $U\rightarrow \calX$ be a smooth affine
  atlas of $\calX$, in particular $U$ has the resolution
  property. It follows that every complex $E^\bullet$ in $\DCoh$
  can be represented, locally in the lisse-\'etale topology on
  $\calX$, as a $2$-terms complex $[\hat{E}^{-1}\rightarrow
  \hat{E}^0]$ such that $\hat{E}^0$ is locally free.
\end{remark}
\begin{theorem}\label{theorem_deligne}
  There is an equivalence of categories \[\sfrac{h^1}{h^0}\colon \Dcoh
  \rightarrow \Ho \AbCones\] between $\Dcoh$ and the homotopy category
  of $\AbCones$, induced by the morphism
\[\hat{h} \colon
{\CCoh}^\opp \rightarrow \AbCones\] such that $\hat{h}([E^{-1}
\rightarrow E^0])= [\sfrac{C(E^{-1})}{C(E^0)}]$.
\end{theorem}
\begin{remark}
  Let $[E^{-1}\rightarrow E^0]$, $[F^{-1}\rightarrow F^0]\in
  \CCoh$.  If $\varkappa\colon E^0 \rightarrow F^{-1}$ is a homotopy
  of morphisms $\psi, \phi\colon E^\bullet\rightarrow F^\bullet$, then
the $2$-morphism $\hat{h}(\varkappa) \colon
  \hat{h}(\psi) \rightarrow \hat{h}(\phi)$ is defined in the following
  way.  For every $\calX$-scheme $U$ and every $(P,f) \in
  \hat{h}(F^\bullet)(U)$, let $\{U_i\}$ be an open cover of $U$ such
  that $U_i \times_U P \cong U_i \times_{\calX} C(F^0)$,
  then \[\hat{h}(\varkappa)(U)(P,f) \colon \hat{h}(\psi)(U)(P,f)
  \rightarrow \hat{h}(\phi)(U)(P,f)\] is obtained by gluing the
  isomorphisms
  \[U_i \times_{\calX} C(E^0) \xrightarrow{(\id_{U_i}, C(\varkappa)
    \circ f_i|_{U_i\times_{\calX}\{0_F\}}\circ p_1 + p_2)} U_i
  \times_{\calX} C(E^0) \text{,}\] where $C(\varkappa)$ is the
  morphism of cones induced by $\varkappa$, $f_i=f|_{U_i\times_U P}$
  and $p_1$, $p_2$ are the natural projections. In particular
  $\hat{h}(\varkappa)$ is a $2$-isomorphism.
\end{remark}
\begin{crit}\label{useful_criterion}
Let $\psi \colon E^\bullet \rightarrow F^\bullet$ be a morphism in
$\CCoh$ that induces an isomorphism $C(F^0) \cong C(F^{-1})_{C(E^{-1})}C(E^{0})$ and a surjective morphism
\[C(d_E) + C(\psi^{-1}) \colon C(E^{-1})
\times_{\calX} C(F^{-1}) \rightarrow C(E^{-1})\text{,}\] then
$\hat{h}(\psi)$ is an isomorphism of cone stacks.
\end{crit}
\begin{proof}
  We can assume that $C(\psi^{-1})$ is surjective (since
  $\hat{h}(F^\bullet) \cong \hat{h}(F^\bullet \oplus
  E^{-1})$). Therefore the statement is implied by the following
  cartesian diagram
  \[
  \begin{tikzpicture}
    \def\x{3.6}
    \def\y{-1.8}
    \node (A0_0) at (0*\x, 0*\y) {$C(F^{-1}) \times_{\calX} C(F^0)$};
    \node (A0_2) at (2*\x, 0*\y) {$C(F^{-1})$};
    \node (A1_0) at (0*\x, 1*\y) {$C(F^{-1}) \times_{\calX} C(E^0)$};
    \node (A1_1) at (1.1*\x, 1*\y) {$C(E^{-1}) \times_{\calX} C(E^0)$};
    \node (A1_2) at (2*\x, 1*\y) {$C(E^{-1})$};
    \node (A2_0) at (0*\x, 2*\y) {$C(F^{-1})$};
    \node (A2_1) at (1.1*\x, 2*\y) {$C(E^{-1})$};
    \node (A2_2) at (2*\x, 2*\y) {$\hat{h}(E^\bullet)$};
    \path (A2_1) edge [->] node [auto] {$\scriptstyle{}$} (A2_2);
    \path (A1_0) edge [->] node [auto] {$\scriptstyle{}$} (A1_1);
    \path (A0_2) edge [->] node [auto] {$\scriptstyle{}$} (A1_2);
    \path (A1_1) edge [->] node [auto] {$\scriptstyle{}$} (A1_2);
    \path (A1_0) edge [->] node [auto] {$\scriptstyle{}$} (A2_0);
    \path (A0_0) edge [->] node [auto] {$\scriptstyle{}$} (A0_2);
    \path (A1_1) edge [->] node [auto] {$\scriptstyle{}$} (A2_1);
    \path (A0_0) edge [->] node [auto] {$\scriptstyle{}$} (A1_0);
    \path (A2_0) edge [->] node [auto] {$\scriptstyle{}$} (A2_1);
    \path (A1_2) edge [->] node [auto] {$\scriptstyle{}$} (A2_2);
  \end{tikzpicture}
  \]
(the upper square is cartesian because $C(F^{0}) = C(F^{-1})
  \times_{C(E^{-1})}C(E^0)$).
\end{proof}
\begin{lemma}\label{lemma_coh}
  The natural functor \[\DCoh \rightarrow
  \Dcoh\] is an equivalence of categories.
\end{lemma}
\begin{proof}
  The statement can be checked locally in the lisse-\'etale topology
  on $\calX$, therefore we can assume that $\calX$ has the resolution
  property. First we show that the functor is fully faithfull. Let
  $E^\bullet, F^\bullet \in \DCoh$, we want to show that the canonical
  map \[\Hom_{\DCoh}(E^\bullet, F^\bullet)\rightarrow
  \Hom_{\Dcoh}(E^\bullet, F^\bullet)\] is a bijection. We can assume
  $E^\bullet= [E^{-1}\xrightarrow{d_E}E^0]$ and $F^\bullet=
  [F^{-1}\xrightarrow{d_F}F^0]$. We can reduce to the case where
  $E^\bullet$ is a coherent sheaf $E$ (similarly $F^\bullet=F$), using
  the following distinguished
  triangle \[E^{-1}\xrightarrow{d_E}E^0\rightarrow
  E^\bullet\xrightarrow{+1}E^{-1}[1]\text{.}\] By resolution property,
  there exists a surjective morphism $\psi \colon P^0 \rightarrow E$
  from a locally free sheaf $P^0$; then $[\ker \psi\rightarrow P^0]$
  is a complex of locally free sheaves quasi-isomorphic to $E$ and
  using the argument above we can reduce to the case where $E^\bullet$
  is a locally free sheaf $E$. Let $E'=\sfrac{E}{\scrO_{\calX}}$, then
  $\rk E'<\rk E$, hence we can reduce to $E=\scrO_{\calX}$. That is,
  we are left to show that \[\Hom_{\DCoh}(\scrO_{\calX},
  F[n])\rightarrow \Hom_{\Dcoh}(\scrO_{\calX}, F[n])\] is a bijection
  for every coherent sheaf $F$ and $n=-1,0$. If $n=-1$, both groups
  are zero. If $n=0$ then both sides are $\sez (\calX, F)$.

  Let us show that every complex $E^\bullet\in
  \Dcoh$ is in the essential image. We
  can assume $E^\bullet= [E^{-1}\xrightarrow{d_E}E^0]$. We have the
  following exact sequence of complexes of sheaves \[0 \rightarrow
  h^{-1}(E^\bullet)[1]\rightarrow E^\bullet \rightarrow [\im d_E
  \rightarrow E^0]\rightarrow 0\text{,}\] which induces a
  distinguished triangle \[h^{-1}(E^\bullet)[1]\rightarrow E^\bullet
  \rightarrow [\im d_E \rightarrow
  E^0]\xrightarrow{+1}h^{-1}(E^\bullet)[2]\text{.}\] Notice that $[\im
  d_E \rightarrow E^0]= h^0(E^\bullet)$ in
  $\Dcoh$. Then we have a distinguished
  triangle \[h^{-1}(E^\bullet)[1]\rightarrow E^\bullet \rightarrow
  h^0(E^\bullet)\xrightarrow{+1}h^{-1}(E^\bullet)[2]\text{.}\] Since
  $h^0(E^\bullet)$ and $h^{-1}(E^\bullet)$ are coherent, the morphism
  $h^0(E^\bullet)[-1]\xrightarrow{+1}h^{-1}(E)[1]$ corresponds to a
  morphism $\psi \colon h^0(E^\bullet)[-1]\rightarrow h^{-1}(E)[1]$ in
  $\DCoh$. Completing $\psi$ to a
  distinguished triangle in $\DCoh$ and
  mapping it to $\Dcoh$, we deduce that
  $E^\bullet$ is quasi-isomorphic to the mapping cone of $\psi$, hence
  it is in the essential image.
\end{proof}
\begin{proof}[Proof of Theorem~\ref{theorem_deligne}]
  We need to prove the following facts: (1) $\psi\colon E^\bullet
  \rightarrow F^\bullet$ is a quasi-isomorphism if and only if
  $\hat{h}(\psi)$ is an isomorphism; (2) if $\psi,\phi \colon
  E^\bullet \rightarrow F^\bullet$ are morphisms of complexes and
  $\xi\colon \hat{h}(\psi)\rightarrow \hat{h}(\phi)$ is an
  $\bbA^1$-equivariant $2$-isomorphism then there exists a unique
  homotopy $\varkappa \colon \psi \rightarrow \phi$ such that
  $\xi=\hat{h}(\varkappa)$; (3) $\sfrac{h^1}{h^0}$ is essentially
  surjective. By Remark~\ref{remark_resolution} and
  Lemma~\ref{lemma_coh},we can assume $E^\bullet = [E^{-1}
  \xrightarrow{d_E} E^0]$ and $F^\bullet = [F^{-1} \xrightarrow{d_F}
  F^0]$ with $E^0$, $F^0$ locally free.

  For the second statement, we define a morphism $C(\varkappa) \colon
  C(F^{-1}) \rightarrow C(E^0)$ as follows. Let $T$ be an
  $\calX$-scheme and let $f \colon T \rightarrow C(F^{-1})$ be a
  morphism. Then $f$ defines an element $(P=T \times_{\calX}
  C(F^0), f_P)$ in $\hat{h}(F^\bullet)(T)$. The images
  $\hat{h}(\psi)(T)(P, f_P)$ and $\hat{h}(\phi)(T)(P, f_P)$ are
  trivial, hence $\xi(T)(P, f_P)$ corresponds to a morphism $g \colon
  T \rightarrow C(E^0)$. We define $C(\varkappa)(T)(f)=g$. Then
  $C(\varkappa)$ induces a homomorphism $\varkappa \colon E^0
  \rightarrow F^{-1}$. Moreover $\xi = \hat{h}(\varkappa)$ and
  $\varkappa$ is unique by construction.

  Let us now prove the first statement. Let $G= E^0
  \times_{F^0}F^{-1}$, then \[0 \rightarrow G \rightarrow E^0 \oplus
  F^{-1} \rightarrow F^0\] is exact. Notice that $E^0 \oplus F^{-1}
  \rightarrow F^0$ is surjective if and only if $h^0(\psi)$ is
  surjective. Let us assume that $h^0(\psi)$ is surjective, then we
  get an exact sequence of cones \[0 \rightarrow C(F^0) \rightarrow
  C(E^0) \times C(F^{-1}) \rightarrow C(G) \rightarrow 0\text{.}\]
  Applying the useful criterion~\ref{useful_criterion}, we obtain
  $[\sfrac{C(F^{-1})}{C(F^0)}] \cong [\sfrac{C(G)}{C(E^0)}]$, hence
  the following diagram
  \[
  \begin{tikzpicture}
    \def\x{2.8}
    \def\y{-1.6}
    \node (A0_0) at (0*\x, 0*\y) {$C(G)$};
    \node (A0_1) at (1*\x, 0*\y) {$C(E^{-1})$};
    \node (A1_0) at (0*\x, 1*\y) {$\sfrac{h^1}{h^0}(F^\bullet)$};
    \node (A1_1) at (1*\x, 1*\y) {$\sfrac{h^1}{h^0}(E^\bullet)$};
    \path (A0_0) edge [->] node [auto] {$\scriptstyle{}$} (A0_1);
    \path (A0_0) edge [->] node [auto] {$\scriptstyle{}$} (A1_0);
    \path (A0_1) edge [->] node [auto] {$\scriptstyle{}$} (A1_1);
    \path (A1_0) edge [->] node [auto] {$\scriptstyle{\sfrac{h^1}{h^0}(\psi)}$} (A1_1);
  \end{tikzpicture}
  \]
  is cartesian and in particular $\sfrac{h^1}{h^0}(\psi)$ is
  representable. If moreover $h^0(\psi)$ is an isomorphism and
  $h^{-1}(\psi)$ is surjective, then the morphism $E^{-1} \rightarrow
  G$ is surjective, hence $C(G) \rightarrow C(E^{-1})$ is a closed
  immersion, which implies that $\sfrac{h^1}{h^0}(\psi)$ is a closed
  immersion. If $h^{-1}(\psi)$ is also an isomorphism then $E^{-1}
  \rightarrow G$ is an isomorphism and so $C(G) \cong C(E^{-1})$; it
  follows that $\sfrac{h^1}{h^0}(\psi)$ is an isomorphism.  Viceversa,
  if $\sfrac{h^1}{h^0}(\psi)$ is representable then the induced
  morphism on automorphisms of objects is injective. Hence we have
  that the morphism \[C(h^0(\psi)) \colon C(h^0(F^\bullet))
  \rightarrow C(h^0(E^\bullet))\] is a closed immersion, which implies
  that $h^0(\psi)$ is surjective.  If moreover
  $\sfrac{h^1}{h^0}(\psi)$ is a closed immersion then $C(G)
  \rightarrow C(E^{-1})$ is a closed immersion, hence $E^{-1}
  \rightarrow G$ is surjective. It follows that $h^0(\psi)$ is
  injective and $h^{-1}(\psi)$ is surjective. Finally, if
  $\sfrac{h^1}{h^0}(\psi)$ is an isomorphism then $C(G) \cong
  C(E^{-1})$, hence $E^{-1} \cong G$, from which we get that
  $h^{-1}(\psi)$ is injective.  It follows that $\sfrac{h^1}{h^0}$ is
  fully faithfull.

  It remains to show that every abelian cone stack $\frakC$ over
  $\calX$ is in the essential image of $\sfrac{h^1}{h^0}$. By
  definition, lisse-\'etale locally on $\calX$, $\frakC\times \cong
  [\sfrac{C(E_U^{-1})}{C(E_U^0)}]$, where $E_U^{-1}$ is a coherent
  sheaf and $E_U^{0}$ is a locally free sheaf. The collection
  ${\{E_U^{-1}\rightarrow E_U^0\}}_U$ defines a complex
  $[E^{-1}\rightarrow E^0]\in \Dcoh$.
\end{proof}
\begin{parag}{\textbf{Cotangent Complex} [\cite{laumon} 17.3,
    \cite{olsson}, \cite{LO}
    2.2.5]\textbf{.}}\label{cotangent_complex}
  Let $f \colon \calX \rightarrow \calY$ be a quasi-compact and
  quasi-separated morphism of algebraic stacks. There exists
  $L_f^\bullet \in D_\qcoh^{\leq 1}(\calX_{\lisse})$ such that
\begin{enumerate}
\item $f$ is of Deligne-Mumford type if and only if $L_f^\bullet \in
  D_\qcoh^{\leq 0}(\calX_{\lisse})$;
\item $f$ is of Deligne-Mumford type and smooth if and only if $L_f^\bullet=\Omega_f$;
\item for every commutative diagram
  \[
  \begin{tikzpicture}
    \def\x{1.5}
    \def\y{-1.2}
    \node (A0_0) at (0*\x, 0*\y) {$\calX'$};
    \node (A0_1) at (1*\x, 0*\y) {$\calY'$};
    \node (A1_0) at (0*\x, 1*\y) {$\calX$};
    \node (A1_1) at (1*\x, 1*\y) {$\calY$};
    \path (A0_0) edge [->] node [auto] {$\scriptstyle{f'}$} (A0_1);
    \path (A0_0) edge [->] node [left] {$\scriptstyle{g}$} (A1_0);
    \path (A0_1) edge [->] node [auto] {$\scriptstyle{h}$} (A1_1);
    \path (A1_0) edge [->] node [auto] {$\scriptstyle{f}$} (A1_1);
  \end{tikzpicture}
  \]
  there exists a morphism $Lg^*L_f^\bullet \rightarrow
  L_{f'}^\bullet$; if the diagram is cartesian and either $h$ of $f$
  is flat, this morphism is an isomorphism;
\item given two morphisms of $S$-stacks $\calX \xrightarrow{f} \calY
  \xrightarrow{g}Z$ with $h=g \circ f$, there exists a natural
  distinguished triangle \[Lf^* L_g^\bullet \rightarrow L_h^\bullet
  \rightarrow L_f^\bullet \rightarrow Lf^*L_g^\bullet[1]\text{;}\]
\item if $f$ factors as $\calX \xrightarrow{i} M \xrightarrow{p} \calY$ with
  $i$ representable and a closed embedding with ideal sheaf $\mathscr{I}$ and $p$ of Deligne-Mumford
  type and smooth, then \[\tau_{\geq -1}L_f^\bullet \cong
  [\sfrac{\mathscr{I}}{\mathscr{I}^2} \rightarrow i^*\Omega_p]\text{.}\]
\end{enumerate}
\end{parag}
\begin{definit}\label{def_normal_sheaf}
  Let $f \colon \calX \rightarrow \calY$ be a morphism of
  algebraic stacks. We define the
  \emph{intrinsic normal sheaf} of $f$ as the abelian cone
  stack $\frakN_f = \sfrac{h^1}{h^0}(\tau_{[-1,0]}L_f^\bullet)$.
\end{definit}
\begin{definit}
A \emph{local embedding} of $f\colon \calX \rightarrow \calY$ is a commutative diagram
  \[
  \begin{tikzpicture}[xscale=1.5,yscale=-1.2]
    \node (A0_0) at (0, 0) {$U$};
    \node (A0_1) at (1, 0) {$\calX$};
    \node (A1_0) at (0, 1) {$M$};
    \node (A1_1) at (1, 1) {$\calY$};
    \path (A0_0) edge [->]node [auto] {$\scriptstyle{u}$} (A0_1);
    \path (A1_0) edge [->]node [auto] {$\scriptstyle{p}$} (A1_1);
    \path (A0_1) edge [->]node [auto] {$\scriptstyle{f}$} (A1_1);
    \path (A0_0) edge [right hook->]node [left] {$\scriptstyle{i}$} (A1_0);
  \end{tikzpicture}
  \]
where $u$, $p$ are smooth and representable and $i$ is a closed embedding.
\end{definit}
\begin{remark}\label{remark_factorization}
  Notice that there exists a local embedding of $f$. More explicitly,
  let $V$ be a smooth atlas for $\calY$ and let $U$ be an affine
  scheme which is a smooth atlas for $\calX \times_{\calY} V$. In
  particular there exists a closed embedding $j \colon U
  \hookrightarrow \bbA^n$. Let us set $M=\bbA_V^n$, then $f\circ u$
  factors as $U \xrightarrow{i} M \xrightarrow{p} \calY$, where $i$ is
  a closed embedding and $p$ is smooth and representable.
  Applying \ref{cotangent_complex} (4)--(5) to $f\circ u=p\circ i$ and snake lemma,
  we get a
  quasi-isomorphism \[\tau_{[-1,0]}Lu^*L_f^\bullet\rightarrow
  \tau_{[-1,0]}[\sfrac{\mathscr{I}}{\mathscr{I}^2}\rightarrow
  i^*\Omega_p\rightarrow \Omega_u]\] ($\mathscr{I}$ being the ideal sheaf corresponding to $i$), which, by
  Theorem~\ref{theorem_deligne}, implies
  \[u^*\frakN_f=\sfrac{h^1}{h^0}(\tau_{[-1,0]}[\sfrac{\mathscr{I}}{\mathscr{I}^2}\rightarrow
  i^*\Omega_p\rightarrow \Omega_u])
  =[\sfrac{N_i}{(\sfrac{i^*T_p}{T_u})}]\text{,}\] where $N_i=
  C(\sfrac{\mathscr{I}}{\mathscr{I}^2})$ is the normal sheaf of $i$.
\end{remark}
\begin{prop}\label{prop_normal_cone}
  There exists a unique closed subcone stack $\frakC_f \subseteq
  \frakN_f$ such that
  \[u^*\frakC_f=[\sfrac{C_i}{(\sfrac{i^*T_p}{T_u})}]\text{,}\] for every local
  embedding of $f$ ($C_i$ being the normal cone of $i$).  If moreover $\calY$ is purely dimensional of pure
  dimension $n$, then $\frakC_f$ is purely dimensional of pure
  dimension $n$.
\end{prop}
\begin{proof}
  By Remark~\ref{remark_factorization}, there exists a local embedding
  $f\circ u=p\circ i$ of $f$ and
  $u^*\frakN_f=[\sfrac{N_i}{(\sfrac{i^*T_p}{T_u})}]$. By \cite{BF}
  3.2, the action of $i^*T_p$ on $N_i$ leaves $C_i$ invariant, hence
  we can define the quotient stack
  $[\sfrac{C_i}{(\sfrac{i^*T_p}{T_u})}]$ which is a closed subcone
  stack of $u^*\frakN_f$ (since $C_i$ is a closed subcone of
  $N_i$). Let us show that $[\sfrac{C_i}{(\sfrac{i^*T_p}{T_u})}]$ does
  not depend on the factorization $p\circ i$ chosen. Let $U
  \xrightarrow{i'} M' \xrightarrow{p'} \calY$ be another factorization
  of $f\circ u$. It is enough to check
  $[\sfrac{C_i}{(\sfrac{i^*T_p}{T_u})}]=[\sfrac{C_j}{(\sfrac{j^*T_{p\circ
        q}}{T_u})}]$ as closed substacks of $u^*\frakN_f$, where $U
  \xrightarrow{j} M'\times_{\calY} M\xrightarrow{q} M$ is the induced
  local embedding. By Remark~\ref{remark_factorization},
  both $N_i$ and $N_j$ are smooth atlases of $u^*\frakN_f$. We get the
  following commutative diagram
  \[
  \begin{tikzpicture}
    \def\x{2.0}
    \def\y{-1.4}
    \node (A0_0) at (0*\x, 0*\y) {$C_j$};
    \node (A1_0) at (0*\x, 1*\y) {$C_i$};
    \node (A0_1) at (1*\x, 0*\y) {$N_j$};
    \node (A0_2) at (2*\x, 0*\y) {$u^*\frakN_f$};
    \node (A1_1) at (1*\x, 1*\y) {$N_i$};
    \path (A0_1) edge [->] node [auto] {$\scriptstyle{\alpha_j}$} (A0_2);
    \path (A0_0) edge [->] node [auto] {$\scriptstyle{}$} (A0_1);
    \path (A1_0) edge [->] node [auto] {$\scriptstyle{}$} (A1_1);
    \path (A0_0) edge [->] node [auto] {$\scriptstyle{}$} (A1_0);
    \path (A1_1) edge [->] node [below] {$\scriptstyle{\alpha_i}$} (A0_2);
    \path (A0_1) edge [->] node [left] {$\scriptstyle{\phi}$} (A1_1);
  \end{tikzpicture}
  \]
  where the left square is cartesian by \cite{fulton},
  Example~4.2.6; therefore it is enough to show that the inverse
  images of $[\sfrac{C_i}{(\sfrac{i^*T_p}{T_u})}]$ and
  $[\sfrac{C_j}{(\sfrac{j^*T_{p\circ q}}{T_u})}]$ in $N_j$ are the
  same. We
  have \[\phi^{-1}(\alpha_i^{-1}([\sfrac{C_i}{(\sfrac{i^*T_p}{T_u})}]))=
  \phi^{-1}(C_i)=C_j = \alpha_j^{-1}([\sfrac{C_j}{(\sfrac{j^*T_{p\circ
        q}}{T_u})}])\text{.}\]

  Let $f\circ u'=p'\circ i'$ be another local embedding of $f$, with
  $u'\colon U'\rightarrow \calX$. We want to show that the cone stacks
  $[\sfrac{C_i}{(\sfrac{i^*T_p}{T_u})}]$ and
  $[\sfrac{C_{i'}}{(\sfrac{i'^*T_{p'}}{T_{u'}})}]$ agree on $V=U
  \times_{\calX} U'$. We can find a closed embedding $j\colon V
  \rightarrow M''$ and smooth morphisms of schemes $q\colon
  M''\rightarrow M$, $q'\colon M''\rightarrow M'$ such that $i\circ
  v=q\circ j$, $i'\circ v'=q'\circ j$, where $v\colon V\rightarrow U$,
  $v'\colon V\rightarrow U'$ are the projections. It is enough to show
  that
  $v^*[\sfrac{C_i}{(\sfrac{i^*T_p}{T_u})}]=[\sfrac{C_j}{(\sfrac{j^*T_{p\circ
        q}}{T_{u\circ v}})}]$.  By
  Remark~\ref{remark_factorization}, \[v^*u^*\frakN_f=[\sfrac{N_j}{(\sfrac{j^*T_{p\circ
        q}}{T_{u\circ
        v}})}]=v^*[\sfrac{N_i}{(\sfrac{i^*T_{p}}{T_{u}})}]\] and the
  following diagram
  \[
  \begin{tikzpicture}
    \def\x{2.0}
    \def\y{-1.4}
    \node (A0_0) at (0*\x, 0*\y) {$N_j$};
    \node (A0_1) at (1*\x, 0*\y) {$v^*u^*\frakN_f$};
    \node (A1_0) at (0*\x, 1*\y) {$v^*N_i$};
    \path (A0_0) edge [->] node [auto] {$\scriptstyle{\alpha_j}$} (A0_1);
    \path (A1_0) edge [->] node [below] {$\scriptstyle{\alpha_i}$} (A0_1);
    \path (A0_0) edge [->] node [left] {$\scriptstyle{\phi}$} (A1_0);
  \end{tikzpicture}
  \]
  is commutative; therefore it is enough to show that the inverse
  images of $v^*[\sfrac{C_i}{(\sfrac{i^*T_p}{T_u})}]$ and
  $[\sfrac{C_j}{(\sfrac{j^*T_{p\circ q}}{T_{u\circ v}})}]$ in $N_j$
  are the same. By Example~4.2.6 in \cite{fulton}, $\phi^{-1}(v^*C_i)=C_j$
  and
  hence \[\phi^{-1}(\alpha_i^{-1}(v^*[\sfrac{C_i}{(\sfrac{i^*T_p}{T_u})}]))=
  \phi^{-1}(v^*C_i)=C_j = \alpha_j^{-1}([\sfrac{C_j}{(\sfrac{j^*T_{p\circ
        q}}{T_{u\circ v}})}])\text{.}\]

  Finally, let us assume that $\calY$ is purely dimensional of pure
  dimension $n$. Let $f\circ u=p\circ i$ be a local embedding of $f$,
  then $u^*\frakC_f = [\sfrac{C_i}{(\sfrac{i^*T_p}{T_u})}]$. We can
  assume that $M$ is purely dimensional of pure dimension $m$ and $p$,
  $u$ are smooth of relative dimension $m-n$ (see the construction in
  Remark~\ref{remark_factorization}). By \cite{fulton}, B.6.6, we have
  that $C_i$ is purely dimensional of pure dimension $m$. Moreover, we
  have the following cartesian diagram
  \[
  \begin{tikzpicture}
    \def\x{2.0}
    \def\y{-1.4}
    \node (A0_0) at (0*\x, 0*\y) {$(\sfrac{i^*T_p}{T_u}) \times C_i$};
    \node (A0_1) at (1*\x, 0*\y) {$C_i$};
    \node (A1_0) at (0*\x, 1*\y) {$C_i$};
    \node (A1_1) at (1*\x, 1*\y) {$[\sfrac{C_i}{(\sfrac{i^*T_p}{T_u})}]$};
    \path (A0_0) edge [->] node [auto] {$\scriptstyle{}$} (A0_1);
    \path (A0_0) edge [->] node [auto] {$\scriptstyle{}$} (A1_0);
    \path (A0_1) edge [->] node [auto] {$\scriptstyle{}$} (A1_1);
    \path (A1_0) edge [->] node [auto] {$\scriptstyle{}$} (A1_1);
  \end{tikzpicture}
  \]
  from which we get that $C_i \rightarrow
  [\sfrac{C_i}{(\sfrac{i^*T_p}{T_u})}]$ is surjective and smooth of
  relative dimension $0$. It follows that $u^*\frakC_f$ is purely
  dimensional of pure dimension $m$ and hence $\frakC_f$ is purely
  dimensional of pure dimension $n$.
\end{proof}
\begin{definit}
The closed subcone stack $\frakC_f$ of $\frakN_f$ is called the
\emph{intrinsic normal cone} of $f$.
\end{definit}
\begin{prop}\label{prop_cone_basechange}
Consider the following commutative diagram of algebraic stacks
  \begin{equation}\label{eq:base_change}
  \begin{tikzpicture}[baseline=-20pt]
    \def\x{1.5}
    \def\y{-1.2}
    \node (A0_0) at (0*\x, 0*\y) {$\calX'$};
    \node (A0_1) at (1*\x, 0*\y) {$\calY'$};
    \node (A1_0) at (0*\x, 1*\y) {$\calX$};
    \node (A1_1) at (1*\x, 1*\y) {$\calY$};
    \path (A0_0) edge [->] node [auto] {$\scriptstyle{f'}$} (A0_1);
    \path (A1_0) edge [->] node [auto] {$\scriptstyle{f}$} (A1_1);
    \path (A0_1) edge [->] node [auto] {$\scriptstyle{h}$} (A1_1);
    \path (A0_0) edge [->] node [left] {$\scriptstyle{g}$} (A1_0);
  \end{tikzpicture}
  \end{equation}
  Then there exists a natural morphism $\alpha \colon \frakC_{f'}
  \rightarrow g^*\frakC_f$ such that
\begin{enumerate}
\item if \eqref{eq:base_change} is cartesian then $\alpha$ is a
  closed immersion;
\item if moreover the morphism $h$ is flat then $\alpha$ is an
  isomorphism.
\end{enumerate}
\end{prop}
\begin{proof}
  Let $U\xrightarrow{u} \calX$ and $U' \xrightarrow{u'}\calX'$ be
  smooth affine atlases with a morphism $v \colon U' \rightarrow U$
  such that $u\circ v=g \circ u'$. There exist a commutative diagram
  \[
  \begin{tikzpicture}
    \def\x{1.8}
    \def\y{-1.4}
    \node (A0_0) at (0*\x, 0*\y) {$U'$};
    \node (A0_1) at (1*\x, 0*\y) {$M'$};
    \node (A0_2) at (2*\x, 0*\y) {$\calY'$};
    \node (A1_0) at (0*\x, 1*\y) {$U$};
    \node (A1_1) at (1*\x, 1*\y) {$M$};
    \node (A1_2) at (2*\x, 1*\y) {$\calY$};
    \path (A0_0) edge [->] node [auto] {$\scriptstyle{i'}$} (A0_1);
    \path (A0_1) edge [->] node [auto] {$\scriptstyle{\widetilde{h}}$} (A1_1);
    \path (A1_0) edge [->] node [auto] {$\scriptstyle{i}$} (A1_1);
    \path (A0_2) edge [->] node [auto] {$\scriptstyle{h}$} (A1_2);
    \path (A1_1) edge [->] node [auto] {$\scriptstyle{p}$} (A1_2);
    \path (A0_0) edge [->] node [auto] {$\scriptstyle{\widetilde{g}}$} (A1_0);
    \path (A0_1) edge [->] node [auto] {$\scriptstyle{p'}$} (A0_2);
  \end{tikzpicture}
  \]
  where $i$, $i'$ are closed embeddings, $p$, $p'$ are representable
  smooth and $M'=M\times_{\calY}\calY'$. If \eqref{eq:base_change} is
  cartesian, we can take $U'=U\times_M M'=U\times_{\calX}\calX'$.

  By~\ref{cotangent_complex} (3), there is an isomorphism $T_{p'}
  \rightarrow q^*T_{p}$, which is an isomorphism $i'^*T_{p'}
  \rightarrow v^*i^*T_{p}$. Moreover there exists a morphism
  $v^*\Omega_u\rightarrow h^0(L_{g\circ u'})\rightarrow \Omega_{u'}$
  which induces a morphism $T_{u'}\rightarrow v^*T_u$; if
  \eqref{eq:base_change} is cartesian, this morphism is an
  isomorphism. By \cite{fulton} B.6, there exists a morphism
  $\widetilde{\alpha} \colon C_{i'} \rightarrow v^*C_i$, induced by
  the natural map $\mathscr{I} \otimes_{\scrO_U}\scrO_{U'} \rightarrow \mathscr{I}'$. If
  \eqref{eq:base_change} is cartesian, $\widetilde{\alpha}$ is a
  closed embedding and, if moreover $h$ is flat, $\widetilde{\alpha}$
  is an isomorphism. Summing up, we get a morphism of stacks
  \[u'^*\frakC_{f'}=[\sfrac{C_{i'}}{(\sfrac{i'^*T_{p'}}{T_{u'}})}]
    \xrightarrow{\alpha}
    v^*[\sfrac{C_i}{(\sfrac{i^*T_p}{T_u})}]=v^*(u^*\frakC_{f})\text{,}\]
    which is a closed immersion if \eqref{eq:base_change} is
    cartesian; if moreover $h$ is flat then $\alpha$ is an isomorphism.
\end{proof}
\section{Perfect obstruction theories}
\begin{definit}
  Let $f\colon \calX\rightarrow \calY$ be a morphism of algebraic
  stacks and let $E^\bullet \in \Onecoh$. A morphism $\phi \colon
  E^\bullet \rightarrow \tau_{\geq -1}L_f^\bullet$ in $\Onecoh$ is
  called an \emph{obstruction theory} for $f$ if $h^1(\phi)$,
  $h^0(\phi)$ are isomorphisms and $h^{-1}(\phi)$ is surjective.
\end{definit}
\begin{remark}\label{remark_obstr_theory}
  If $(E^\bullet, \phi)$ is an obstruction theory for $f$
  then, by the proof of Theorem~\ref{theorem_deligne}, the morphism
  $\sfrac{h^1}{h^0}(\phi) \colon \frakN_f \rightarrow
  \sfrac{h^1}{h^0}(\tau_{[-1,0]}E^\bullet)$ is a closed immersion.
\end{remark}
\begin{parag}{\textbf{Deformation theory of morphisms of algebraic stacks.}}\label{section_deformation}
  For the basic definitions of deformation theory we refer to
  \cite{talpo}. Let $f \colon \calX \rightarrow \calY$ be a morphism
  of algebraic stacks over a base scheme $S$. Let $\spec \bark
  \xrightarrow{\barx} \calX$ be a geometric point of $\calX$.  Let
  $\Lambda = \hat{\scrO}_{S,\barx}$. Consider the deformation category
  $h_{\calX, \barx}$ such that, for all $A \in
  (\sfrac{\Art}{\Lambda})$, the objects of $h_{\calX, \barx}(A)$ are
  morphisms $g_{\calX} \colon \spec A \rightarrow \calX$ such that
  $g_{\calX}|_{\spec\bark}= \barx$. There is a natural functor $\nu_f
  \colon h_{\calX, \barx} \rightarrow h_{\calY, \barx}$ given by the
  composition with $f$.  Let $g_{\calX}\in h_{\calX, \barx}(A)$,
  $g_{\calY}'\in h_{\calY, \barx}(A')$ be such that
  $g_{\calY}'(i)=f\circ g_{\calX}$, where $i\colon \spec A \rightarrow
  \spec A'$. We denote by $\mathcal{S}_{f}$ the set of isomorphism
  classes of $g_{\calX}'\in h_{\calX, \barx}(A')$ such that
  $g_{\calX}'\circ i=g_{\calX}$, $f\circ g_{\calX}'=g_{\calY}'$.
\end{parag}
\begin{theorem}\label{theorem_deff}
  Let $L_f^\bullet$ be the cotangent complex of $f$. Then,
  for every geometric point $\barx$ of $\calX$ and for every small
  extension $A' \rightarrow A=\sfrac{A'}{I}$ in
  $(\sfrac{\Art}{\Lambda})$,
\begin{enumerate}
\item there is a functorial surjective set-theoretical map \[\ob_f
  \colon h_{\calX, \barx}(A) \times_{h_{\calY, \barx}(A)} h_{\calY, \barx}(A')
  \rightarrow h^1((L\barx^*L_f^\bullet)\spcheck) \otimes I\] such that
  $\ob_f(g_{\calX}, g_{\calY}')=0$ if and only if $\mathcal{S}_{f}$ is not empty;
\item if $\ob_f(g_{\calX}, g_{\calY}')=0$ then $\mathcal{S}_{f}$ is a torsor under
  $h^0((L\barx^*L_f^\bullet)\spcheck)\otimes I$;
\item if $\ob_f(g_{\calX}, g_{\calY}')=0$ and $g_{\calX}'\in \mathcal{S}_{f}$, then the group of
  infinitesimal authomorphisms of $g_{\calX}'$ with respect to $(g_{\calX},g_{\calY}')$
  is isomorphic to $h^{-1}((L\barx^*L_f^\bullet)\spcheck)\otimes I$.
\end{enumerate}
\end{theorem}
\begin{proof}
  Let us consider a local embedding $u \colon U\rightarrow \calX$ of
  $f$ (see Remark~\ref{remark_factorization}), then the map $\barx
  \rightarrow \calX$ factors through $u$. Let us set $f_U=f\circ
  u$. By~\ref{cotangent_complex},
  $h^1((L\barx^*L_f^\bullet)\spcheck)\cong
  h^1((L\barx^*L_{f_U}^\bullet)\spcheck)$. We have the following exact
  sequence \[0 \rightarrow
  h^{-1}((L\barx^*L_f^\bullet)\spcheck)\rightarrow
  h^{0}((L\barx^*L_u^\bullet)\spcheck)\rightarrow
  h^{0}((L\barx^*L_{f_U}^\bullet)\spcheck)\rightarrow
  h^{0}((L\barx^*L_f^\bullet)\spcheck)\rightarrow 0\text{.}\]
  Moreover, by deformation theory of schemes, we know that there is a
  functorial exact sequence \[0 \rightarrow
  h^0((L\barx^*L_{f_U}^\bullet)\spcheck) \otimes I \rightarrow h_{U,
    \barx}(A') \rightarrow h_{U, \barx}(A)\times_{h_{\calY,
      \barx}(A)}h_{\calY, \barx}(A') \xrightarrow{\ob_{f_U}}
  h^1((L\barx^*L_{f_U}^\bullet)\spcheck) \otimes I \rightarrow
  0\text{.}\] For every $g_{\calX} \in h_{\calX, \barx}(A)$, there
  exists $g_{U,1} \in h_{U, \barx}(A)$ such that $u \circ
  g_{U,1}=g_{\calX}$. If $g_{U,2}\in h_{U, \barx}(A)$ is another
  morphism such that $u \circ g_{U,2}=g_{\calX}$ then there exists a
  unique morphism $g_{U\times_{\calX} U}\in h_{U\times_{\calX} U,
    \barx}(A)$ such that $u_j \circ g_{U\times_{\calX} U}=g_{U,j}$ for
  $j=1,2$, where $u_j \colon U \times_{\calX} U \rightarrow U$ are the
  projections. By~\ref{cotangent_complex}, $u_1$ and $u_2$ induces
  isomorphisms $h^1((L\barx^*L_{f_{U\times_{\calX}
      U}}^\bullet)\spcheck)\cong
  h^1((L\barx^*L_{f_U}^\bullet)\spcheck)$. Therefore \[\ob_{f_U}(g_{U,j},g_{\calY}')=\ob_{f_U}(u_j\circ
  g_{U \times_{\calX} U},g_{\calY}')= \ob_{f_{U \times_{\calX}
      U}}(g_{U \times_{\calX} U},g_{\calY}')\text{.}\] Hence, setting
  $\ob_f(g_{\calX},g_{\calY}')=\ob_{f_U}(g_{U,1},g_{\calY}')$, we get
  a well-defined surjective map \[\ob_f \colon h_{\calX,
    \barx}(A)\times_{h_{\calY, \barx}(A)}h_{\calY, \barx}(A')
  \rightarrow h^1((L\barx^*L_f^\bullet)\spcheck) \otimes I\text{.}\]
  Notice that $\ob_f(g_{\calX},g_{\calY}')=0$ if and only if
  $\ob_{f_U}(g_U,g_{\calY}')=0$ for some $g_U\in h_{U, \barx}(A)$ such
  that $u\circ g_U=g_{\calX}$. If $g_U'\in h_{U, \barx}(A')$ is such
  that $g_U'\circ i=g_U$, $f_U\circ g_U'=g_{\calY}'$, then setting
  $g_{\calX}'=u\circ g_U'$ we have $g_{\calX}'\circ i=g_{\calX}$,
  $f\circ g_{\calX}=g_{\calY}'$; viceversa, if $g_{\calX}'\in
  h_{\calX, \barx}(A')$ is such that $g_{\calX}'\circ i=g_{\calX}$,
  $f\circ g_{\calX}'=g_{\calY}'$, then there exists $g_U$ such that
  $u\circ g_U=g_{\calX}\circ i$, hence there exists $g_U'$ such that
  $g_U'\circ i=g_U$, $u\circ g_U'=g_{\calX}'$ and $f_U\circ
  g_U'=g_{\calY}'$.

  By~\ref{cotangent_complex}, we have the following commutative
  diagram with exact rows
  \[
  \begin{tikzpicture}[xscale=2.8,yscale=-1.6]
    \node (A0_0) at (0, 0) {$0$};
    \node (A0_1) at (0.7, 0) {$h^{-1}((L\barx^*L_f^\bullet)\spcheck)$};
    \node (A0_2) at (1.8, 0) {$h^0((L\barx^*L_{\widetilde{u}}^\bullet)\spcheck)$};
    \node (A0_3) at (3.0, 0) {$h^0((L\barx^*L_{f_{U \times_{\calX} U}}^\bullet)\spcheck)$};
    \node (A0_4) at (4.3, 0) {$h^0((L\barx^*L_f^\bullet)\spcheck)$};
    \node (A0_5) at (5, 0) {$0$};
    \node (A1_0) at (0, 1) {$0$};
    \node (A1_1) at (0.7, 1) {$h^{-1}((L\barx^*L_f^\bullet)\spcheck)$};
    \node (A1_2) at (1.8, 1) {$h^0((L\barx^*L_u^\bullet)\spcheck)$};
    \node (A1_3) at (3.0, 1) {$h^0((L\barx^*L_{f_U}^\bullet)\spcheck)$};
    \node (A1_4) at (4.3, 1) {$h^0((L\barx^*L_f^\bullet)\spcheck)$};
    \node (A1_5) at (5, 1) {$0$};
    \path (A1_4) edge [->]node [auto] {$\scriptstyle{}$} (A1_5);
    \path (A0_3) edge [->]node [auto] {$\scriptstyle{\sigma_{\widetilde{u}}}$} (A0_4);
    \path (A0_1) edge [-,double distance=1.5pt]node [auto] {$\scriptstyle{}$} (A1_1);
    \path (A1_0) edge [->]node [auto] {$\scriptstyle{}$} (A1_1);
    \path (A0_3) edge [->>]node [auto] {$\scriptstyle{\rho_{u_j}}$} (A1_3);
    \path (A0_1) edge [->]node [auto] {$\scriptstyle{}$} (A0_2);
    \path (A0_2) edge [->>]node [auto] {$\scriptstyle{\rho_{u_j}}$} (A1_2);
    \path (A0_4) edge [->]node [auto] {$\scriptstyle{}$} (A0_5);
    \path (A0_0) edge [->]node [auto] {$\scriptstyle{}$} (A0_1);
    \path (A0_4) edge [-,double distance=1.5pt]node [auto] {$\scriptstyle{}$} (A1_4);
    \path (A1_1) edge [->]node [auto] {$\scriptstyle{}$} (A1_2);
    \path (A1_2) edge [->]node [auto] {$\scriptstyle{\rho_u}$} (A1_3);
    \path (A0_2) edge [->]node [auto] {$\scriptstyle{\rho_{\widetilde{u}}}$} (A0_3);
    \path (A1_3) edge [->]node [auto] {$\scriptstyle{\sigma_u}$} (A1_4);
  \end{tikzpicture}
  \]
  where $\widetilde{u}=u\circ u_j$. Let us fix $\overline{g}_{\calX}
  \in h_{\calX,\barx}(A)$, $\overline{g}_{\calY}' \in
  h_{\calY,\barx}(A')$, $\overline{g}_U \in h_{U,\barx}(A)$,
  $\overline{g}_{U\times_{\calX} U} \in h_{U\times_{\calX}
    U,\barx}(A)$ such that $f\circ
  \overline{g}_{\calX}=\overline{g}_{\calY}'\circ i$, $u\circ
  \overline{g}_U=\overline{g}_{\calX}$, $u_j \circ
  \overline{g}_{U\times_{\calX} U}=\overline{g}_{U}$. Let us assume
  that $\ob_f(\overline{g}_{\calX},\overline{g}_{\calY}')=0$ and fix
  $\overline{g}_{\calX}' \in \mathcal{S}_{f}$, $\overline{g}_U' \in
  \mathcal{S}_{f_U}$, $\overline{g}_{U\times_{\calX} U}' \in
  \mathcal{S}_{f_{U\times_{\calX} U}}$ such that $u \circ
  \overline{g}_U'=\overline{g}_{\calX}'$, $u_j \circ
  \overline{g}_{U\times_{\calX} U}'=\overline{g}_U'$. There is a
  natural surjective map $\mathcal{S}_{f_U}\rightarrow \mathcal{S}_f$
  given by composition with $u$. If $\alpha_{f_U,1}$,
  $\alpha_{f_U,2}\in h^0((L\barx^*L_{f_U}^\bullet)\spcheck)\otimes I$
  are such that $\sigma_u(\alpha_{f_U,1})=\sigma_u(\alpha_{f_U,2})$,
  then there exists $\alpha_{f_{U\times_{\calX}U}}\in
  h^0((L\barx^*L_{f_{U\times_{\calX}U}}^\bullet)\spcheck)\otimes I$
  such that
  $\rho_{u_j}(\alpha_{f_{U\times_{\calX}U}})=\alpha_{f_U,j}$. Therefore \[u\circ
  (\alpha_{f_U,j}\cdot
  \overline{g}_{U}')=u\circ\left(u_j(\alpha_{f_{U\times_{\calX}U}})\cdot
    (u_j\circ\overline{g}_{U\times_{\calX}U})\right)=\widetilde{u}\circ
  (\alpha_{f_{U\times_{\calX}U}}\cdot
  \overline{g}_{U\times_{\calX}U})\text{.}\] Hence, given $\alpha_f\in
  h^0((L\barx^*L_{f}^\bullet)\spcheck)\otimes I$ and setting
  $\alpha_f\cdot \overline{g}_{\calX}'=u\circ (\alpha_{f_U}\cdot
  \overline{g}_{U}')$ for some $\alpha_{f_U}\in
  h^0((L\barx^*L_{f_U}^\bullet)\spcheck)\otimes I$ such that
  $\sigma_u(\alpha_{f_U})=\alpha_{f}$, we get a well-defined
  action. If $g_{\calX}'\in \mathcal{S}_{f}$, there exists $g_{U}'\in
  \mathcal{S}_{f_U}$ such that $u\circ g_U'=g_{\calX}'$; hence there
  exists $\alpha_{f_U}\in
  h^0((L\barx^*L_{f_U}^\bullet)\spcheck)\otimes I$ such that
  $g_U'=\alpha_{f_U}\cdot \overline{g}_U'$. It follows that
  $g_{\calX}'=\sigma_u(\alpha_{f_U})\cdot
  \overline{g}_{\calX}'$. Moreover, if $\alpha_f\cdot
  \overline{g}_{\calX}'=\overline{g}_{\calX}'$, then there exists
  $g_{U\times_{\calX}U}'\in \mathcal{S}_{f_{U\times_{\calX}U}}$ such
  that $u_1\circ g_{U\times_{\calX}U}'=\overline{g}_U'$, $u_2\circ
  g_{U\times_{\calX}U}'=\alpha_{f_U}\cdot \overline{g}_U'$ for some
  $\alpha_{f_U}\in h^0((L\barx^*L_{f_U}^\bullet)\spcheck)\otimes I$
  such that $\sigma_u(\alpha_{f_U})=\alpha_{f}$. Hence there exists
  $\alpha_{f_{U\times_{\calX}U}}\in
  h^0((L\barx^*L_{f_{U\times_{\calX}U}}^\bullet)\spcheck)\otimes I$
  such that $\rho_{u_1}(\alpha_{f_{U\times_{\calX}U}})=e_{f_U}$,
  $\rho_{u_2}(\alpha_{f_{U\times_{\calX}U}})=\alpha_{f_U}$ ($e_{f_U}$
  being the
  identity). Therefore \[\alpha_f=\sigma_u(\alpha_{f_U})=\sigma_U\circ
  \rho_{u_2}(\alpha_{f_{U\times_{\calX}U}})=\sigma_u \circ
  \rho_{u_1}(\alpha_{f_{U\times_{\calX}U}})=\sigma_u(e_{f_U})=e_f\text{.}\]

  Finally, let $g_{\calX} \in h_{\calX,\barx}(A)$, $g_{\calY}' \in
  h_{\calY,\barx}(A)$ such that $f\circ g_{\calX}=g_{\calY}'\circ i
  =g_{\calY}$, $\ob_f(g_{\calX},g_{\calY}')=0$ and let us fix
  $g_{\calX}'\in \mathcal{S}_f$. We claim that the kernel of the
  natural homomorphism \[\eta\colon \aut_{\calX}(g_{\calX}')\rightarrow
  \aut_{\calX}(g_{\calX})\times_{\aut_{\calY}(g_{\calY})}\aut_{\calY}(g_{\calY}')\]
  is isomorphic to $h^{-1}((L\barx^*L_f^\bullet)\spcheck)\otimes
  I$. Let consider $g_U\in h_{U,\barx}(A)$, $g_U'\in h_{U,\barx}(A')$
  such that $g_U'\circ i =g_U$, $u\circ g_U'=g_{\calX}'$. By
  deformation theory of schemes, we have the following commutative
  diagram with exact rows
  \[
  \begin{tikzpicture}[xscale=4.6,yscale=-2.0]
    \node (A0_0) at (0, 0) {$\aut_U(g_U')$};
    \node (A0_1) at (0.9, 0) {$\aut_U(g_U)\times_{\aut_{\calX}(g_{\calX})} \aut_{\calX}(g_{\calX}')$};
    \node (A0_2) at (2, 0) {$h^0((L\barx^*L_u^\bullet)\spcheck)\otimes I$};
    \node (A1_0) at (0, 1) {$\aut_U(g_U')$};
    \node (A1_1) at (0.9, 1) {$\aut_U(g_U)\times_{\aut_{\calY}(g_{\calY})} \aut_{\calY}(g_{\calY}')$};
    \node (A1_2) at (2, 1) {$h^0((L\barx^*L_{f_U}^\bullet)\spcheck)\otimes I$};
    \path (A0_0) edge [->]node [auto] {$\scriptstyle{}$} (A0_1);
    \path (A0_1) edge [->]node [auto] {$\scriptstyle{\gamma}$} (A1_1);
    \path (A1_0) edge [->]node [auto] {$\scriptstyle{}$} (A1_1);
    \path (A0_2) edge [->]node [auto] {$\scriptstyle{\rho_u}$} (A1_2);
    \path (A1_1) edge [->]node [auto] {$\scriptstyle{\omega_{f_U}}$} (A1_2);
    \path (A0_0) edge [-,double distance=1.5pt]node [auto] {$\scriptstyle{}$} (A1_0);
    \path (A0_1) edge [->]node [auto] {$\scriptstyle{\omega_u}$} (A0_2);
  \end{tikzpicture}
  \]
  where $\aut_U(g_U')=\aut_U(g_U)=\{\id\}$ since $U$ is a
  scheme. Moreover, $\alpha_{f_U}$ (respectively $\alpha_u$) acts
  trivially on $\mathcal{S}_{f_U}$ (respectively $\mathcal{S}_u$) if
  and only if it is in the image of $\omega_{f_U}$ (respectively
  $\omega_u$). It follows that $\ker
  \gamma=h^{-1}((L\barx^*L_X^\bullet)\spcheck)\otimes I$. The claim
  follows therefore by noticing that $\ker \eta=\ker\gamma$.
\end{proof}
\begin{theorem}[Infinitesimal criterion for obstruction theories]\label{theorem_critot}
  A pair $(E^\bullet, \phi)$ is an obstruction theory for $f$ if and
  only if, for every geometric point $\barx$ of $\calX$ and for every
  small extension $A' \rightarrow A=\sfrac{A'}{I}$ in
  $(\sfrac{\Art}{\Lambda})$,
\begin{enumerate}
\item the obstruction
  $h^1(\phi\spcheck)(\ob_f(g_{\calX},g_{\calY}'))\in
  h^1((L\barx^*E^\bullet)\spcheck)\otimes I$ vanishes if and only if there
  exists a morphism $g_{\calX}'$ such that $g_{\calX}'\circ i
  =g_{\calX}$, $f\circ g_{\calX}'=g_{\calY}'$,
\item if $h^1(\phi\spcheck)(\ob_f(g_{\calX},g_{\calY}'))=0$ then the set of
  isomorphism classes of such morphisms $g_{\calX}'$ is a torsor under
  $h^0((L\barx^*E^\bullet)\spcheck)\otimes I$,
\item if $\ob_f(g_{\calX}, g_{\calY}')=0$ and $g_{\calX}'\in h_{\calX, \barx}(A')$ is such
  that $f \circ g_{\calX}'=g_{\calY}'$, $g_{\calX}'\circ i=g_{\calX}$, then the group of
  infinitesimal authomorphisms of $g_{\calX}'$ with respect to $(g_{\calX},g_{\calY}')$
  is isomorphic to $h^{-1}((L\barx^*E^\bullet)\spcheck)\otimes I$.
\end{enumerate}
\end{theorem}
\begin{proof}
  If $(E^\bullet, \phi)$ is an obstruction theory for $f$, the
  statement follows immediately from
  Theorem~\ref{theorem_deff}. Viceversa, let assume that the second
  part of the statement holds and let show that $h^0(\phi)$,
  $h^1(\phi)$ are isomorphisms and $h^{-1}(\phi)$ is surjective. Since
  the statement is lisse-\'etale local, we can assume that $\calX$ is an affine
  scheme $\spec R$. Then, by assumptions, for every $R$-algebra $B$
  and $B$-module $N$, there is a bijection
  $\hom(h^j(L_f^\bullet)\otimes B,N)\rightarrow
  \hom(h^j(E^\bullet)\otimes B,N)$ for $j=0,1$, which implies that $h^j(\phi)$ is
  an isomorphism for $j=0,1$. We can assume that $f$ factors as
  $\calX\xrightarrow{i} M\xrightarrow{p}\calY$ with $i$ a closed
  embedding with ideal sheaf $\mathscr{I}$ into an affine scheme $M$
  and $p$ smooth. We can further assume that $\ker (E^0\rightarrow E^1)$ is locally free,
  $E^{-1}$ is a coherent sheaf, $E^i = 0$ for $i \neq 1, 0, -1$ and
  $\phi^{-1}$ surjective. Then, by Remark~\ref{remark_factorization}, the complex
  $G\rightarrow i^*\Omega_p\rightarrow \Omega_u$, where $G$ is the cokernel of $\ker
  \phi^0 \times_{E^0}E^{-1}\rightarrow E^{-1}$, is quasi-isomorphic to
  $E^\bullet$. Therefore we can assume $E^0=i^*\Omega_p$, $E^1=\Omega_u$ and we have
  to prove that $E^{-1} \rightarrow
  \sfrac{\mathscr{I}}{\mathscr{I}^2}$ is surjective; let $F$ be its
  image. Let $\calX=\spec A$, $F'\subset \mathscr{I}$ the inverse
  image of $F$, and $\spec A' \subset M$ the subscheme defined by
  $F'$; let $g \colon \spec A \rightarrow \calX$ be the identity. We
  can extend $g$ to the inclusion $g' \colon \spec A' \rightarrow
  M$. Let $\pi\colon \sfrac{\mathscr{I}}{\mathscr{I}^2}\rightarrow
  \sfrac{\mathscr{I}}{F'}$ be the natural projection. By assumption
  $\pi$ factors via $E^0$ if and only if $g$ extends to a map $\spec
  A'\rightarrow \calX$, if and only if $\pi \circ \phi^{-1}\colon
  E^{-1}\rightarrow\sfrac{\mathscr{I}}{F'}$ factors via $E^0$. As $\pi
  \circ \phi^{-1}$ is the zero map, it certainly factors. Therefore
  $\pi$ also factors. Moreover, the fact that $\pi$ factors via $E^0$
  together with $\pi \circ \phi^{-1}= 0$ implies $\pi = 0$, hence
  $\phi^{-1}\colon E^{-1} \rightarrow
  \sfrac{\mathscr{I}}{\mathscr{I}^2}$ is surjective.
\end{proof}
\begin{definit}
Let $(E^\bullet, \phi)$ be an obstruction theory for $f$. We
say that $(E^\bullet, \phi)$ is \emph{perfect} (of perfect amplitude
contained in $[-1,0]$) if, lisse-\'etale locally over $\calX$, it is isomorphic
to $[E^{-1} \rightarrow E^0\rightarrow E^1]$ with $E^{-1}$, $\ker(E^0\rightarrow E^1)$ locally free
sheaves over $\calX$.
\end{definit}
\begin{remark}
  An obstruction theory $(E^\bullet, \phi)$ is perfect if and only if
  $\sfrac{h^1}{h^0}(\tau_{[-1,0]}E^\bullet)$ is a vector bundle stack
  over $\calX$.
\end{remark}
\section{Virtual fundamental class}\label{subsection_virtual}
\begin{parag}
  Let $f \colon \calX \rightarrow \calY$ be a morphism of algebraic
  stacks, let us assume that $\calY$ is purely dimensional of pure
  dimension $m$ and that $\calX$ admits a stratification by global
  quotients in the sense of \cite{kresch} 3.5.3. Let $(E^\bullet,
  \phi)$ be a perfect obstruction theory for $f$, we denote
  by \[\mu \colon \frakE_f= \sfrac{h^1}{h^0}(\tau_{[-1,0]}E^\bullet) \rightarrow
  \calX\] the associated vector bundle stack of rank $r$. By
  Remark~\ref{remark_obstr_theory}, the intrinsic normal cone
  $\frakC_f$ is a closed substack of $\frakE_f$.  Moreover, by
  \cite{kresch} 4.3.2, the flat pullback \[\mu^* \colon A_*(\calX)
  \rightarrow A_{*+r}(\frakE_f)\] is an isomorphism and we denote the
  inverse by $0^!$.
\end{parag}
\begin{definit}
The \emph{virtual fundamental class} of $\calX$ relative to $(E^\bullet,
\phi)$ is the cycle class \[{[\calX, E^\bullet]}^\virt=  0^![\frakC_f]\in
A_*(\calX)\text{.}\]
\end{definit}
\begin{remark}
The intrinsic cone $\frakC_f$ is purely dimensional of pure dimension
$m$, therefore ${[\calX, E^\bullet]}^\virt \in A_{m -
  r}(\calX)$ and $m - r$ is called the \emph{virtual dimension}
of $\calX$.
\end{remark}
\begin{prop}[Base-change]\label{prop_virt_basechange}
Consider the following cartesian diagram of algebraic stacks
  \begin{equation}\label{eq:base_change_virt}
  \begin{tikzpicture}[baseline=-20pt]
    \def\x{1.5}
    \def\y{-1.2}
    \node (A0_0) at (0*\x, 0*\y) {$\calX'$};
    \node (A0_1) at (1*\x, 0*\y) {$\calY'$};
    \node (A1_0) at (0*\x, 1*\y) {$\calX$};
    \node (A1_1) at (1*\x, 1*\y) {$\calY$};
    \path (A0_0) edge [->] node [auto] {$\scriptstyle{f'}$} (A0_1);
    \path (A1_0) edge [->] node [auto] {$\scriptstyle{f}$} (A1_1);
    \path (A0_1) edge [->] node [auto] {$\scriptstyle{h}$} (A1_1);
    \path (A0_0) edge [->] node [left] {$\scriptstyle{g}$} (A1_0);
  \end{tikzpicture}
  \end{equation}
  where $\calY$ and $\calY'$ are smooth and purely dimensional of pure
  dimension $m$, $\calX$ and $\calX'$ admit stratifications by global
  quotients. Let $(E^\bullet, \phi)$ be a perfect obstruction
  theory for $f$. If $h$ is flat or a regular local immersion (of
  constant dimension) then \[h^!{[\calX, E^\bullet]}^\virt={[\calX',
    Lg^*E^\bullet]}^\virt\text{.}\]
\end{prop}
\begin{proof}
  Let us notice that $Lg^* E^\bullet$ is a perfect obstruction theory
  for $f'$. Let $\frakE_{f'}= \sfrac{h^1}{h^0}(Lg^* E^\bullet)$ and
  let $0'^!$ the inverse of $\mu'^*$, where $\mu' \colon \frakE_{f'}
  \rightarrow \calX'$. If $h$ is flat then, by
  Proposition~\ref{prop_cone_basechange}, we have $g^* \frakC_f \cong
  \frakC_{f'}$, hence $h^![\frakC_f]=[\frakC_{f'}]\in
  A_*(\frakE_{f'})$. Therefore we get
\[h^!{[\calX, E^\bullet]}^\virt = h^! 0^! [\frakC_f]= 0'^!
h^![\frakC_f]= 0'^! [\frakC_{f'}]= {[\calX',
  Lg^*E^\bullet]}^\virt\text{.}\] If $h$ is a regular local immersion,
let consider $\rho\colon \frakN_h\rightarrow \calY'$ and let
$\widetilde{0}\colon g^*\frakC_f\rightarrow
\frakN_h\times_{\calY}\frakC_f$ be the zero section. Then
$\widetilde{0}^![\frakC_{\sfrac{g^*\frakC_f}{\frakC_f}}]=h^![\frakC_f]\in A_*(g^*\frakC_f)$,
by definition of $h^!$, and
\[\widetilde{0}^![\rho^*\frakC_{f'}]=\widetilde{0}^!\rho^*[\frakC_{f'}]=[\frakC_{f'}]\in A_*(g^*\frakC_f)\text{.}\]
Moreover, by \cite{BF} 3.3--3.5,
$[\frakC_{\sfrac{g^*\frakC_f}{\frakC_f}}]=[\rho^*\frakC_{f'}]\in A_*(\frakN_h\times_{\calY}\frakC_f)$. Hence
$h^![\frakC_f]=[\frakC_{f'}]\in A_*(\frakE_{f'})$ and one concludes as before.
\end{proof}
\begin{parag}
  Let us consider the cartesian diagram \eqref{eq:base_change_virt}
  and assume that $h$ is a local complete intersection morphism of
  stacks with finite unramified diagonal over $\calY$. Let $E^\bullet$
  and $E'^\bullet$ be perfect obstruction theories for $f$ and $f'$
  respectively. Then $E^\bullet$ and $E'^\bullet$ are
  \emph{compatible} over $h$ if there exists a homomorphism of
  distinguished triangles in $\Dcohprime$
  \begin{equation}\label{eq:compatible}
  \begin{tikzpicture}[baseline=-20pt]
    \def\x{1.6}
    \def\y{-1.2}
    \node (A0_0) at (0*\x, 0*\y) {$g^*E^\bullet$};
    \node (A0_1) at (1*\x, 0*\y) {$E'^\bullet$};
    \node (A0_2) at (2*\x, 0*\y) {$f'^*L_h^\bullet$};
    \node (A0_3) at (3*\x, 0*\y) {$g^*E^\bullet[1]$};
    \node (A1_0) at (0*\x, 1*\y) {$g^*L_{\calX}^\bullet$};
    \node (A1_1) at (1*\x, 1*\y) {$L_{\calX'}^\bullet$};
    \node (A1_2) at (2*\x, 1*\y) {$L_g^\bullet$};
    \node (A1_3) at (3*\x, 1*\y) {$g^*L_{\calX}^\bullet[1]$};
    \path (A1_2) edge [->]node [auto] {$\scriptstyle{}$} (A1_3);
    \path (A0_0) edge [->]node [auto] {$\scriptstyle{}$} (A0_1);
    \path (A0_1) edge [->]node [auto] {$\scriptstyle{}$} (A0_2);
    \path (A1_0) edge [->]node [auto] {$\scriptstyle{}$} (A1_1);
    \path (A0_3) edge [->]node [auto] {$\scriptstyle{}$} (A1_3);
    \path (A0_2) edge [->]node [auto] {$\scriptstyle{}$} (A1_2);
    \path (A1_1) edge [->]node [auto] {$\scriptstyle{}$} (A1_2);
    \path (A0_0) edge [->]node [auto] {$\scriptstyle{}$} (A1_0);
    \path (A0_1) edge [->]node [auto] {$\scriptstyle{}$} (A1_1);
    \path (A0_2) edge [->]node [auto] {$\scriptstyle{}$} (A0_3);
  \end{tikzpicture}
  \end{equation}
\end{parag}
\begin{prop}[Functoriality]\label{prop_virtfunctorial}
  Let $E^\bullet$ and $E'^\bullet$ be compatible perfect obstruction
  theories as above. If either
  $h$ is smooth or $\calY$ and $\calY'$ are smooth, then \[h^!{[\calX,
    E^\bullet]}^\virt={[\calX', E'^\bullet]}^\virt\text{.}\] 
\end{prop}
\begin{proof}
  By \cite{BF} 2.7, the diagram~\eqref{eq:compatible} induces a short
  exact sequence of vector bundle stacks \[f'^*\frakN_h\rightarrow
  \frakE_{f'}\xrightarrow{\psi} g^*\frakE_f\] If $h$ is
  smooth, by \cite{BF} 3.14,
  $\frakC_{f'}=g^*\frakC_{f}\times_{g^*\frakE_{f}}\frakE_{f'}$, hence
  $0_{g^*\frakE_f}^![g^*\frakC_f]=0_{\frakE_{f'}}^![\frakC_{f'}]\in A_*(\calX')$ and
  \[h^!{[\calX,E^\bullet]}^\virt=h^!0_{\frakE_f}^![\frakC_f]
  =0_{g^*\frakE_f}^![g^*\frakC_f]={[\calX',E'^\bullet]}^\virt\text{.}\]
  If $\calY$ and $\calY'$ are smooth then $h$ factors as
  $\calY'\xrightarrow{\Gamma_h}\calY'\times_S\calY\xrightarrow{p}
  \calY$, where $\Gamma_h$ is the graph of $h$, which is a regular
  local immersion, and $p$ is smooth. We have the following cartesian
  diagram
  \[
  \begin{tikzpicture}[xscale=2.0,yscale=-1.4]
    \node (A0_0) at (0, 0) {$\calX'$};
    \node (A0_1) at (1, 0) {$\calY'\times_S \calX$};
    \node (A0_2) at (2, 0) {$\calX$};
    \node (A1_0) at (0, 1) {$\calY'$};
    \node (A1_1) at (1, 1) {$\calY'\times_S \calY$};
    \node (A1_2) at (2, 1) {$\calY$};
    \path (A0_0) edge [->]node [auto] {$\scriptstyle{}$} (A0_1);
    \path (A0_1) edge [->]node [auto] {$\scriptstyle{}$} (A0_2);
    \path (A1_0) edge [->]node [auto] {$\scriptstyle{\Gamma_h}$} (A1_1);
    \path (A0_2) edge [->]node [auto] {$\scriptstyle{f}$} (A1_2);
    \path (A1_1) edge [->]node [auto] {$\scriptstyle{p}$} (A1_2);
    \path (A0_0) edge [->]node [left] {$\scriptstyle{f'}$} (A1_0);
    \path (A0_1) edge [->]node [auto] {$\scriptstyle{\widetilde{f}}$} (A1_1);
  \end{tikzpicture}
  \]
  and we consider the obstruction theory $\Omega_{\calY'}\oplus
  E^\bullet$ for $\widetilde{f}$. Notice that
  $\Omega_{\calY'}\oplus E^\bullet$ is compatible with
  $E^\bullet$ over $p$ and with $E'^\bullet$ over $\Gamma_h$. Then, by
  the first part, we can assume that $h$ is a regular local
  immersion. By \cite{BF} 5.9,
  $[f'^*\frakN_h\times_{\calX'}\frakC_{f'}]=
  \psi^*[\frakC_{\sfrac{g^*\frakC_f}{\frakC_f}}]\in A_*(f'^*\frakN_h\times_{\calX'}\frakE_{f'})$ and
  hence
\begin{align*}
{[\calX',E'^\bullet]}^\virt&=0_{\frakE_{f'}}^![\frakC_{f'}]=
  0_{f'^*\frakN_h\times_{\calX'}\frakE_{f'}}^![f'^*\frakN_h\times_{\calX'}\frakC_{f'}]=
  0_{f'^*\frakN_h\times_{\calX'}\frakE_{f'}}^!\psi^*[\frakC_{\sfrac{g^*\frakC_f}{\frakC_f}}]\\ &= 0_{f'^*\frakN_h\times_{\calX'}g^*\frakE_{f}}^![\frakC_{\sfrac{g^*\frakC_f}{\frakC_f}}]= 0_{\frakC_f\times_{\calY}\frakN_h}^![\rho^*\frakC_{f'}]=
  0_{g^*\frakE_f}^!h^![\frakC_f]=h^!0_{\frakE_f}^![\frakC_f]\\ &=h^!{[\calX,
    E^\bullet]}^\virt\text{.}\qedhere
\end{align*}
\end{proof}
\section{Virtual pullbacks and Costello's pushforward}\label{section_int}
\begin{parag}
Although not mentioned in~\cite{kresch}, one can verify that the
theory can be extended to Artin stacks over a Dedekind domain. As a
consequence we get that Manolache's construction of the virtual
pullback in~\cite{manolache} is valid for morphisms of Artin stacks
over a Dedekind domain.
\end{parag}
\begin{prop}[Costello's pushforward formula]\label{prop_costellopf}
Let us consider a cartesian diagram
  \[
  \begin{tikzpicture}[xscale=1.5,yscale=-1.2]
    \node (A0_0) at (0, 0) {$\calX_1$};
    \node (A0_1) at (1, 0) {$\calX_2$};
    \node (A1_0) at (0, 1) {$\calY_1$};
    \node (A1_1) at (1, 1) {$\calY_2$};
    \path (A0_0) edge [->]node [auto] {$\scriptstyle{f}$} (A0_1);
    \path (A1_0) edge [->]node [auto] {$\scriptstyle{g}$} (A1_1);
    \path (A0_1) edge [->]node [auto] {$\scriptstyle{p_2}$} (A1_1);
    \path (A0_0) edge [->]node [left] {$\scriptstyle{p_1}$} (A1_0);
  \end{tikzpicture}
  \]
where
\begin{enumerate}
\item $\calY_1,\calY_2$ are Artin stacks over $D$ of the same pure dimension,
\item $\calX_1,\calX_2$ are Artin stacks over $D$ with quasi-finite diagonal,
\item $g$ is a morphism of degree $d$, $f$ is proper,
\item for $i=1,2$, $p_i$ admits perfect obstruction theory $E_i^\bullet$ such that $f^*E_2^\bullet\cong E_1^\bullet$. 
\end{enumerate}
Then \[f_*{[\calX_1, E_1^\bullet]}^{\virt}=d{[\calX_2,
  E_2^\bullet]}^{\virt}\] in each of the following cases
\begin{itemize}
\item[$(a)$] $g$ is projective,
\item[$(b)$] $\calY_1,\calY_2$ are Deligne-Mumford stacks and $g$ is proper,
\item[$(c)$] $\calY_1,\calY_2$ have quasi-finite diagonal and $g$ is proper.
\end{itemize}
\end{prop}
\begin{proof}
  Since in each of the cases listed above we are able to pushforward
  along $g$ (see \cite{EHKV} 2.8 for case (c)), the statement follows by the same argument of
  \cite{manolache} 5.29, after noticing that non-representable proper
  pushforward commutes with virtual pullback (this can be shown in the
  same way as in \cite{manolache} 4.1).
\end{proof}

{\it E-mail address:} {\tt flavia.poma@gmail.com}

\end{document}